 \documentclass[final,1p,times]{elsarticle}
\usepackage{amssymb}
\usepackage{amsthm}
\usepackage{hyperref}
\usepackage{amsmath,amssymb,amsopn,amsfonts,mathrsfs,amsbsy,amscd}
\usepackage{longtable}
\usepackage{color}
\journal{Aequationes Mathematicae}

\newcommand{\e}{\epsilon}

\newtheorem{definition}{Definition}[section]
\newtheorem{theorem}{Theorem}[section]
\newtheorem{proposition}{Proposition}[section]
\newtheorem{lemma}{Lemma}[section]
\newtheorem{corollary}{Corollary}[section]

\newtheorem{remark}{Remark}[section]
\numberwithin{equation}{section}

\begin{document}
\begin{frontmatter}
\title{$\frac{1}{2}$ derivations and biderivations of 3-dimensional nilpotent Lie-Yamaguti algebras  }

\author[label1]{Hassan Oubba}
%\author[label2]{Hassan Oubba}
\address[label1]{Universit\'e Hassan II de Casablanca\\
		Facult\'e des sciences A\"in Chock -  D\'epartement de Math\'ematiques et Informatique\\
		B.P. 5366 Maarif, Casablanca, Maroc\\e-mail:
		 HASSAN.OUBBA@univh2c.ma}
	%\address[label2]{Universit\'e Moulay Isma\"{i}l\\
		%Facult\'e des sciences -  D\'epartement de Mth\'ematiques\\
		%B.P. 11201, Zitoune, M\'ekn\`es 50000, Maroc\\e-mail:
		 %hassan.oubba@edu.umi.ac.ma
	%}

\begin{abstract}
In this paper, we investigate $\frac{1}{2}$-derivations and their corresponding local and 2-local 
analogues on three-dimensional nilpotent Lie--Yamaguti algebras over a field $\mathbb{F}$ 
of characteristic zero. Utilizing the established classification of these algebras, we explicitly 
determine their structures of $\frac{1}{2}$-derivations and biderivations. Furthermore, we 
characterize the matrix representations of their local $\frac{1}{2}$-derivations. Our findings 
demonstrate that, for the non-abelian algebras under consideration, the space of local 
$\frac{1}{2}$-derivations is strictly larger than that of ordinary $\frac{1}{2}$-derivations. 
In addition, we construct examples of 2-local $\frac{1}{2}$-derivations that fail to be global 
$\frac{1}{2}$-derivations. Finally, a complete characterization of the biderivations for these 
algebras is provided.

\end{abstract}
\end{frontmatter}

\section*{Introduction}
Lie–Yamaguti algebras, often abbreviated as LY-algebras, constitute a class of algebraic systems that extend the notions of both Lie algebras and Lie triple systems. The structure of these algebras is defined by a pair of operations: a binary bracket \([\cdot, \cdot]\) and a ternary bracket \(\{\cdot, \cdot, \cdot\}\). Under this framework, an LY-algebra collapses to a standard Lie algebra if the ternary operation vanishes identically. Conversely, in the absence of the binary product, the structure identifies with a Lie triple system, an object of fundamental importance in the geometry of symmetric spaces.

In recent years, local and 2-local operators have attracted much attention in the
theory of non-associative algebras, including Lie, Jordan, and Leibniz algebras. The
notion of local derivations was introduced by Kadison ~\cite{Kad} and by Larson and Sourour
~\cite{Lar}. Later, in 1997, Semrl introduced the notions of 2-local derivations and 2-local
automorphisms ~\cite{Sem}.A related problem for three-dimensional nilpotent Lie-Yamaguti algebras was considered in ~\cite{Mos} by B. Mosbahi and Y. Bakhtiyor. The authors explicitly computed their derivation algebras and further characterized the matrix forms of their local derivations. Their results demonstrate that, for the non-abelian algebras under consideration, the classes of local derivations are strictly larger than those of ordinary derivations.

Let $A$ be an algebra (not necessarily associative). Recall that a linear map $D \colon A \to A$ is called a $\frac{1}{2}$-derivation if 
$D(ab) = \frac{1}{2} \bigl( D(a)b + aD(b) \bigr)$
for all $a, b \in A$. A linear map $\Delta \colon A \to A$ is termed a local $\frac{1}{2}$-derivation if, for every $x \in A$, there exists a $\frac{1}{2}$-derivation $D_x$ of $A$ such that $\Delta(x) = D_x(x)$. Analogously, a $2$-local $\frac{1}{2}$-derivation is a mapping $\Delta \colon A \to A$ (not necessarily linear) such that, for any pair of elements $x, y \in A$, there exists a $\frac{1}{2}$-derivation $D_{x,y}$ of $A$ satisfying $\Delta(x) = D_{x,y}(x)$ and $\Delta(y) = D_{x,y}(y)$. The study of local and $2$-local $\delta$-derivations was initiated in~\cite{khud}, where local and $2$-local $\frac{1}{2}$-derivations of finite-dimensional solvable Lie algebras with filiform, Heisenberg, and abelian nilradicals were investigated. Later, the study of local and 2-local $\frac{1}{2}$-derivations have been investigated in the context of Lie, Leibniz algebras ~\cite{Yus, Yus1, Yus2, mam}.

The concept of biderivations originally emerged within the framework of associative ring theory ~\cite{br0, br}. More specifically, previous literature demonstrated that any biderivation operating on a noncommutative prime ring is necessarily inner. These maps provide a powerful framework for examining commuting mappings; consequently, several authors have successfully characterized commuting maps by initially determining the full structure of biderivations on the respective algebras, as documented in~\cite{bn, br1, Chen, du, ob1, Han, Gua, Tang, Wang1, Wang2}. In this paper, we introduce and study the notion of biderivations for Lie–Yamaguti algebras.

In this paper we investigate $\frac{1}{2}$-derivations, local $\frac{1}{2}$-derivations  of Lie algebras related to three-dimensional nilpotent Lie-Yamaguti algebras. We first obtain
an explicit description of the general form of $\frac{1}{2}$-derivations on these algebras; see
Theorem~\ref{t1}. Section~3 is devoted to local and
2-local $\frac{1}{2}$-derivations. In Section~4, we study biderivations of  three-dimensional nilpotent Lie-Yamaguti algebras.

\section{Preliminaries}
In this section, we recall the basic definitions and notation that will be used throughout the paper. We first
give the definition of a Lie–Yamaguti algebra and recall the notions of derivations, $\frac{1}{2}$-derivations local $\frac{1}{2}$-derivations and
2-local $\frac{1}{2}$-derivations. We also introduce biderivations
in the setting of Lie–Yamaguti algebras. Finally, we recall the classification of three-dimensional nilpotent
Lie–Yamaguti algebras and its derivations, which will serve as the main algebraic framework for the computations carried
out in the subsequent sections.
\begin{definition}
    A \textbf{Lie–Yamaguti algebra} is a vector space $L$ over a field $F$ endowed with a bilinear multiplication $[-,-] \colon L\times L\to L$ and a trilinear multiplication $\{\cdot, \cdot, \cdot\} \colon L\times L \times L \to L$ satisfying, for all $x, y, z, u, v \in L$, the following identities:
\begin{align*}
\text{(LY1)} \quad & [x, y] + [y, x] = 0; \\
\text{(LY2)} \quad & \{x, y, z\} + \{y, x, z\} = 0; \\
\text{(LY3)} \quad & \{x, y, z\} + \{y, z, x\} + \{z, x, y\} + [[x, y], z] + [[y, z], x] + [[z, x], y] = 0; \\
\text{(LY4)} \quad & \{[x, y], z, u\} + \{[y, z], x, u\} + \{[z, x], y, u\} = 0; \\
\text{(LY5)} \quad & \{x, y, [u, v]\} = [\{x, y, u\}, v] + [u, \{x, y, v\}]; \\
\text{(LY6)} \quad & \{u, v, \{x, y, z\}\} = \{\{u, v, x\}, y, z\} + \{x, \{u, v, y\}, z\} + \{x, y, \{u, v, z\}\}.
\end{align*}
\end{definition}
\begin{definition}
    Let $(L, [\cdot, \cdot], \{\cdot, \cdot, \cdot\})$ be a Lie-Yamaguti (LY) algebra. A linear map $D \colon L \to L$ is called a \textbf{derivation} of $L$ if it satisfies:
\begin{align*}
D([x, y]) &= [D(x), y] + [x, D(y)], \\
D(\{x, y, z\}) &= \{D(x), y, z\} + \{x, D(y), z\} + \{x, y, D(z)\}
\end{align*}
for all $x, y, z \in L$. That is, $D$ is simultaneously a derivation with respect to the binary and ternary operations of $L$.
\end{definition}
The set of all derivations of $L$ is denoted by $Der(L)$. The space $Der(L)$ is a subalgebra of $\mathfrak{gl}(L)$ with respect to the commutator bracket.
\begin{definition}
    Let $(L, [\cdot, \cdot], \{\cdot, \cdot, \cdot\})$ be a Lie-Yamaguti (LY) algebra. A linear map $D \colon L \to L$ is called a \textbf{derivation} of $L$ if it satisfies:
\begin{align*}
D([x, y]) &= \frac{1}{2}([D(x), y] + [x, D(y)]), \\
D(\{x, y, z\}) &= \frac{1}{2}(\{D(x), y, z\} + \{x, D(y), z\} + \{x, y, D(z)\})
\end{align*}
for all $x, y, z \in L$. That is, $D$ is simultaneously a $\frac{1}{2}$derivation with respect to the binary and ternary operations of $L$.
\end{definition}
The set of all $\frac{1}{2}$-derivations of $L$ is denoted by $\frac{1}{2}$-Der(L). The space $\frac{1}{2}-Der(L)$ is a subalgebra of $\mathfrak{gl}(L)$ with respect to the commutator bracket.
\begin{definition}
    A linear map $\Delta \colon L \to L$ is said to be a \textbf{local $\frac{1}{2}$-derivation} if for any $x \in L$, there exists a $\frac{1}{2}$-derivation $D_x$ of $L$ such that
\[
\Delta(x) = D_x(x).
\]
\end{definition}
\begin{definition}
    A (not necessarily linear) map $\Delta \colon L \to L$ is said to be a \textbf{2-local derivation} if for any $x, y \in L$, there exists a derivation $D_{x,y}$ of $L$ such that
\[
\Delta(x) = D_{x,y}(x) \quad \text{and} \quad \Delta(y) = D_{x,y}(y).
\]
\end{definition}
\begin{definition}
    Let $(L, [\cdot, \cdot], \{\cdot, \cdot, \cdot\})$ be a Lie-Yamaguti (LY) algebra. A bilinear map $\delta \colon L\times L \to L$ is called a \textbf{biderivation} of $L$ if it satisfies:
\begin{align*}
\delta([x, y],u) &= [\delta(x,u), y] + [x, \delta(y,u)], \\
\delta(u,[x, y]) &= [\delta(u,x), y] + [x, \delta(u,y)], \\
\delta(\{x, y, z\},u) &= \{\delta(x,u), y, z\} + \{x, \delta(y,u), z\} + \{x, y, \delta(z,u)\}),\\
\delta(u,\{x, y, z\}) &= \{\delta(u,x), y, z\} + \{x, \delta(u,y), z\} + \{x, y, \delta(u,z)\})
\end{align*}
for all $x, y, z,u \in L$. That is, $\delta$ is simultaneously a biderivation with respect to the binary and ternary operations of $L$.
\end{definition}
A biderivation $\delta \in BDer(L)$ is termed \textbf{symmetric} if $\delta(x, y) = \delta(y, x)$ holds for all $x, y \in L$; conversely, it is designated as \textbf{skew-symmetric} if $\delta(x, y) = -\delta(y, x)$ for all $x, y \in L$. We denote the subspaces of all symmetric biderivations and all skew-symmetric biderivations on $L$ as $BDer_+(L)$ and $BDer_-(L)$, respectively.
\begin{remark}
    \begin{enumerate}
    \item  It is elementary to check that a bilinear map $\phi : L \times L \rightarrow L$ is a biderivation if and only if the linear maps
\begin{eqnarray*}
    \phi_x=:\phi(x,\cdot) :  L \rightarrow L,\, y\mapsto \phi_x(y):=\phi(x,y)\\
     \phi_y=:\phi(\cdot,y) :  L \rightarrow L,\, x\mapsto \phi_y(x):=\phi(x,y)\\
\end{eqnarray*}
 are both derivations of $\mathfrak{L}$ for any $x,y \in \mathfrak{L}$.
    \item The set of all biderivations on $L$ is a vector space.
\end{enumerate}
\end{remark}
Denote the trivial 1-dimensional Lie-Yamaguti algebra by $L_{1,1}$ and the trivial 2-dimensional Lie-Yamaguti algebra by $L_{2,1}$.

\begin{theorem}~\cite{Ab}
Let $L$ be a nilpotent Lie-Yamaguti algebra of dimension 3. Then $L$ is isomorphic to one of the following algebras:
\begin{itemize}
    \item $L_{3,1} = L_{2,1} \oplus L_{1,1}$ (trivial algebra),
    \item $L_{3,2} \colon [e_1, e_2] = e_3$ (Lie algebra),
    \item $L_{3,3} \colon \{e_1, e_2, e_1\} = e_3$ (Lie triple system),
    \item $L_{3,4} \colon [e_1, e_2] = e_3, \ \{e_1, e_2, e_1\} = e_3$.
\end{itemize}
\end{theorem}
\begin{theorem}~\cite{Mos}\label{tt0}
The derivation algebras of the three-dimensional nilpotent Lie–Yamaguti algebras are as follows:
\begin{align*}
Der(L_{3,1}) &= \left\{ \begin{pmatrix} a_{11} & a_{12} & a_{13} \\ a_{21} & a_{22} & a_{23} \\ a_{31} & a_{32} & a_{33} \end{pmatrix} \;\middle|\; a_{ij} \in F \right\}, \\[2ex]
Der(L_{3,2}) &= \left\{ \begin{pmatrix} a & b & 0 \\ c & d & 0 \\ p & q & a + d \end{pmatrix} \;\middle|\; a, b, c, d, p, q \in F \right\}, \\[2ex]
Der(L_{3,3}) &= \left\{ \begin{pmatrix} a & 0 & 0 \\ b & d & 0 \\ p & q & 2a + d \end{pmatrix} \;\middle|\; a, b, d, p, q \in F \right\},
\end{align*}%
and
\[
Der(L_{3,4}) = \left\{ \begin{pmatrix} 0 & 0 & 0 \\ c & d & 0 \\ p & q & d \end{pmatrix} \;\middle|\; c, d, p, q \in F \right\}.
\]
\end{theorem}

\section{Algorithm for $\frac{1}{2}$-derivations}
Let $\{e_1, e_2, \dots, e_n\}$ be a basis of an $n$-dimensional Lie-Yamaguti algebra $(L, [\cdot, \cdot], \{\cdot, \cdot, \cdot\})$. The binary and ternary multiplications are defined by the structure constants $\gamma_{ij}^k, \theta_{ijk}^l \in F$ such that
\[
[e_i, e_j] = \sum_{k=1}^n \gamma_{ij}^k e_k, \quad \{e_i, e_j, e_k\} = \sum_{l=1}^n \theta_{ijk}^l e_l, \quad 1 \le i, j, k \le n.
\]
A linear map $D \in End(L)$ is represented by a matrix $(\alpha_{pq})$ where
\[
D(e_q) = \sum_{p=1}^n \alpha_{pq} e_p.
\]
The $\frac{1}{2}$-derivation condition requires:
\begin{align*}
D([e_i, e_j]) &= \frac{1}{2}([D(e_i), e_j] + [e_i, D(e_j)]), \\
D(\{e_i, e_j, e_k\}) &=\frac{1}{2}( \{D(e_i), e_j, e_k\} + \{e_i, D(e_j), e_k\} + \{e_i, e_j, D(e_k)\}).
\end{align*}
Using the structure constants, these conditions become the following systems of equations:

For the binary bracket:
\[
\sum_{r=1}^n \gamma_{ij}^r \alpha_{tr} = \frac{1}{2}\left(\sum_{p=1}^n \left( \alpha_{pi} \gamma_{pj}^t + \alpha_{pj} \gamma_{ip}^t \right)\right), \quad \forall 1 \le i, j, t \le n.
\]

For the ternary product:
\[
\sum_{s=1}^n \theta_{ijk}^s \alpha_{ts} =\frac{1}{2}\left( \sum_{p=1}^n \left( \alpha_{pi} \theta_{pjk}^t + \alpha_{pj} \theta_{ipk}^t + \alpha_{pk} \theta_{ijp}^t \right)\right), \quad \forall 1 \le i, j, k, t \le n.
\]
Solving this combined system of linear equations provides all derivations of the Lie-Yamaguti algebra $L$.
\begin{theorem}\label{t1}
     The $\frac{1}{2}$-derivation algebras of the three-dimensional nilpotent Lie–Yamaguti algebras are as
follows:
\begin{align*}
\frac{1}{2}Der(L_{3,1}) &= \left\{ \begin{pmatrix} a_{11} & a_{12} & a_{13} \\ a_{21} & a_{22} & a_{23} \\ a_{31} & a_{32} & a_{33} \end{pmatrix} \;\middle|\; a_{ij} \in F \right\}, \\[2ex]
\frac{1}{2}Der(L_{3,2}) &= \left\{ \begin{pmatrix} a & b & 0 \\ c & d & 0 \\ p & q & a + d \end{pmatrix} \;\middle|\; a, b, c, d, p, q \in F \right\}, \\[2ex]
\frac{1}{2}Der(L_{3,3}) &= \left\{ \begin{pmatrix} a & 0 & 0 \\ b & d & 0 \\ p & q & 2a + d \end{pmatrix} \;\middle|\; a, b, d, p, q \in F \right\},
\end{align*}%
and
\[
\frac{1}{2}Der(L_{3,4}) = \left\{ \begin{pmatrix} 0 & 0 & 0 \\ c & d & 0 \\ p & q & d \end{pmatrix} \;\middle|\; c, d, p, q \in F \right\}.
\]
\end{theorem}
\section{Local and 2-local $\frac{1}{2}$-derivations}
In this section, we study local and 2-local $\frac{1}{2}$-derivations of three-dimensional nilpotent Lie–Yamaguti algebras. Using the explicit descriptions of ordinary $\frac{1}{2}$-derivations obtained in the previous section, we
determine the matrix forms of all local $\frac{1}{2}$-derivations for each algebra $L_{3,i}$, $i = 1, 2, 3, 4$. We show that,
for the non-abelian algebras under consideration, the space of local $\frac{1}{2}$-derivations is strictly larger than the corresponding $\frac{1}{2}$-derivation algebra. Moreover, we construct examples of 2-local $\frac{1}{2}$-derivations which are not $\frac{1}{2}$-derivations. These results show that, in contrast to some rigid classes of Lie-type algebras,
three-dimensional nilpotent Lie–Yamaguti algebras admit genuinely local mappings which cannot be
represented by a single global $\frac{1}{2}$-derivation.
\begin{theorem}
    Let $\Delta$ be a linear operator on $L_{3,i}$. Then the local $\frac{1}{2}$-derivations of the three-dimensional
nilpotent Lie–Yamaguti algebras are described as follows:
\begin{enumerate}
\item $\Delta$ is a local derivation on $L_{3,1}$ if and only if
$$\Delta\in End(L_{3,1}).$$
\item $\Delta$ is a local derivation on $L_{3,2}$ if and only if its matrix has the form
$$\Delta= \begin{pmatrix} a_{11} & a_{12} & 0 \\ a_{21} & a_{22} & 0 \\ a_{31} & a_{32} & a_{33} \end{pmatrix}, \quad a_{ij} \in \mathbb{F} ,$$
\item $\Delta$ is a local derivation on $L_{3,3}$ if and only if its matrix has the form
$$\Delta=  \begin{pmatrix} b_{11} & 0 & 0 \\ b_{21} & b_{22} & 0 \\ b_{31} & b_{32} & b_{33} \end{pmatrix}, \quad b_{ij} \in \mathbb{F} ,$$
\item $\Delta$ is a local derivation on $L_{3,4}$ if and only if its matrix has the form
\[
\Delta=  \begin{pmatrix} 0 & 0 & 0 \\ c_{21} & c_{22} & 0 \\ c_{31} & c_{32} & c_{33} \end{pmatrix} \quad c_{ij} \in \mathbb{F}.
\]
\end{enumerate}
\end{theorem}
\begin{proof}
 For the abelian algebra $L_{3,1}$, all binary and ternary products vanish identically. Consequently, every linear operator acting on 
$L_{3,1}$ constitutes a derivation, which immediately yields:
\[
\text{Loc}{\frac{1}{2}}\text{Der}(L_{3,1}) = \frac{1}{2}\text{Der}(L_{3,1}) = \text{End}(L_{3,1}).
\]

Consider a local derivation $\Delta$ acting on $L_{3,2}$, and let 
$B$ denote its matrix representation with respect to the standard basis:
$$B=\begin{pmatrix} a_{11} & a_{12} & a_{13} \\ a_{21} & a_{22} & a_{23} \\ a_{31} & a_{32} & a_{33} \end{pmatrix}$$
According to the definition of a local $\frac{1}{2}$-derivation, for each 
element $x \in L_{3,2}$, there corresponds a global 
$\frac{1}{2}$-derivation $D_x \in \frac{1}{2}\text{Der}(L_{3,2})$ 
such that the local operator $\Delta$ satisfies:
\begin{equation*}
\Delta(x) = D_x(x).
\end{equation*}
By Theorem ~\ref{t1}, the matrix of $D_x$ has the form
$$D_x=\begin{pmatrix} a_x & b_x & 0 \\ c_x & d_x & 0 \\ p_x & q_x & a_x + d_x \end{pmatrix}.$$
Taking $x = e_3$, we have
$$\Delta(e_3)=D_{e_3}(e_3)\in \mathbb{F}e_3.$$
This implies that $a_{13} = a_{23} = 0$, thereby showing that the matrix 
representation of $\Delta$ reduces to the following form:
$$B=\begin{pmatrix} a_{11} & a_{12} & 0 \\ a_{21} & a_{22} & 0 \\ a_{31} & a_{32} & a_{33} \end{pmatrix}$$
Conversely, assuming that $\Delta$ is a linear operator of this form and let 
$$x=x_1e_1+x_2e_2+x_3e_3.$$
It suffices to demonstrate the existence of a $\frac{1}{2}$-derivation $D_x \in \frac{1}{2}Der(L_{3,2})$ satisfying
$$\Delta(x) = D_x(x).$$
In terms of their matrix representations, this identity reduces to the following system of equations:
\begin{equation}\label{s0}
\begin{cases}
    a_xx_1+b_xx_2=a_{11}x_1+a_{12}x_2,\\
    c_xx_1+d_xx_2=a_{21}x_1+a_{22}x_2,\\
    p_xx_1+q_xx_2+(a_x+d_x)x_3=a_{31}x_1+a_{32}x_2+a_{33}x_3.
\end{cases}
\end{equation}
If $x_1\neq0$, we may choose
$$a_x=a_{11},\quad b_x=a_{12}, \quad c_x=a_{21}, \quad d_x=a_{22}, \quad q_x=0,$$
and choose $$p_x=a_{31}+\frac{a_{31}x_2+(a_{33}-a_{11}-a_{22})x_3}{x_1}.$$
Therefore, the system (~\ref{s0}) is verified.\\
If $x_1= 0$ and $x_2\neq 0$, we may choose
$$a_x=0,\quad b_x=a_{12}, \quad c_x=0, \quad d_x=a_{22}, \quad p_x=0,$$
and taking $$q_x=a_{32}+\frac{(a_{33}-a_{22})x_3}{x_2}.$$
Then again, the system (~\ref{s0}) is verified.\\
If $x_1=x_2=0$, then $x=x_3e_3$. We take $$a_x=a_{33}, \quad d_x=0,$$
while the remaining parameters can be chosen arbitrarily. Consequently,
$$D_x(x)=a_{33}x_3e_3=\Delta(x).$$
Hence, every operator of the aforementioned form is a local $\frac{1}{2}$-derivation on $L_{3,2}$.\\
An analogous argument applies to the cases $L_{3,3}$ and $L_{3,4}$. This concludes the proof.   
\end{proof}
\begin{corollary}
    The following table summarizes the dimensions of the $\frac{1}{2}$-derivation algebras and the spaces of local $\frac{1}{2}$-derivations of $3$-dimensional nilpotent Lie–Yamaguti algebras:
$$\begin{tabular}{lcccc}
\hline
Algebras & $L_{3,1}$ & $L_{3,2}$ & $L_{3,3}$ & $L_{3,4}$ \\
\hline
$\dim \text{Der}(L_{3,i})$ & 9 & 6 & 5 & 4 \\
\hline
$\dim \text{LocDer}(L_{3,i})$ & 9 & 7 & 6 & 5 \\
\hline
\end{tabular}$$
\end{corollary}
We consider the example of a 2-local derivation that is not a derivation, as introduced by Mosbahi et al. in \cite{Mos}, and we show that this example is also a 2-local $\frac{1}{2}$-derivation which is not $\frac{1}{2}$-derivation.

\begin{theorem}
    Each of the algebras $L_{3,i}$ (for $i \in \{1, 2, 3, 4\}$) admits a $2$-local $\frac{1}{2}$-derivation which is not a $\frac{1}{2}$-derivation.
\end{theorem}
\begin{proof}
    We define a homogeneous, non-additive function $\phi$ on $\mathbb{F}^2$ as follows:
$$\phi(x_1, x_2) = 
\begin{cases} 
\frac{x_1^2}{x_2}, & \text{if } x_2 \neq 0, \\ 
0, & \text{if } x_2 = 0, 
\end{cases}$$
where $(x_1,x_2)\in \mathbb{F}^2$. This function is homogeneous, that is,
$$\phi(\lambda x_1,\lambda x_2)=\lambda \phi(x_1,x_2),$$
for all $\lambda\in \mathbb{F}$, but it is not additive.\\
For $i = 1, 2, 3, 4$, we define the map $\Delta \colon L_{3,i} \to L_{3,i}$ as follows:
$$\Delta(x)=\phi(x_1,x_2)e_3, \quad \forall x=\sum_{i=1}^3x_ie_i\in L_{3,i}.$$
Since the mapping $\Delta$ is non-linear, it cannot be a derivation.\\
New, we show that $\Delta$ constitutes a $2$-local $\frac{1}{2}$-derivation. Let $x, y$ be arbitrary elements of $L_{3,i}$, written in the form
\[
x = \sum_{j=1}^3 x_j e_j, \quad y = \sum_{j=1}^3 y_j e_j.
\]
Our goal is to construct a derivation $D_{x,y} \in \frac{1}{2}\text{Der}(L_{3,i})$ such that
$$\Delta(x)=D_{x,y}(x)\quad \mbox{and} \quad \Delta(y)=D_{x,y}(y).$$
Let $D_{x,y}$ be the linear operator defined by
$$D_{x,y}(e_1)=pe_3, \quad D_{x,y}(e_2)=qe_3, \quad D_{x,y}(e_3)=0.$$
Its matrix is
$$D_{x,y}=\begin{pmatrix}
    0&0&0\\
    0&0&0\\
    p&q&0
\end{pmatrix}.$$
According to Theorem ~\ref{t1}, this matrix constitutes a $\frac{1}{2}$-derivation matrix for all the algebras $L_{3,1}$, $L_{3,2}$, $L_{3,3}$, and $L_{3,4}$.\\
Consequently, it suffices to select $p, q \in \mathbb{F}$ satisfying
$$D_{x,y}(x)=\Delta(x), \quad D_{x,y}(y)=\Delta(y).$$
These relations can be rewritten as the following system:
$$\begin{cases}
    px_1+qx_2=\phi(x_1,x_2),\\
    py_1+qy_2=\phi(y_1,y_2).
\end{cases}$$
If $x_1 y_2 - x_2 y_1 \neq 0$, then the system above uniquely determines the parameters $p$ and $q$.\\
In the case where $x_1 y_2 - x_2 y_1 = 0$, the linear dependence of the vectors $(x_1, x_2)$ and $(y_1, y_2)$, combined with the homogeneity of $\phi$, guarantees the compatibility of the system. Hence, at least one solution exists.\\

Therefore, for every pair $x, y \in L_{3,i}$, there exists a $\frac{1}{2}$-derivation $D_{x,y}$ such that
$$D_{x,y}(x)=\Delta(x), \quad D_{x,y}(y)=\Delta(y).$$
Consequently, $\Delta$ constitutes a $2$-local $\frac{1}{2}$-derivation on $L_{3,i}$ for each $i \in \{1, 2, 3, 4\}$. Due to its non-linearity, $\Delta$ fails to be a $\frac{1}{2}$-derivation. This establishes that every algebra $L_{3,1}$, $L_{3,2}$, $L_{3,3}$, and $L_{3,4}$ admits a $2$-local $\frac{1}{2}$-derivation that is not a $\frac{1}{2}$-derivation.
\end{proof}
\section{biderivations}
In this section, we determine all biderivations of 3-dimensional nilpotent Lie-Yamaguti algebras.\\

The key structural property utilized throughout this section is that if $\delta$ is a biderivation of a Lie–Yamaguti algebra $L$, then for any fixed $a \in L$, both mappings $x \mapsto \delta(a, x)$ and $x \mapsto \delta(x, a)$ are ordinary derivations of $L$. This justifies the application of Theorem~\ref{tt0} to explicitly determine the matrices $D_{e_i}$.

For any $\delta \in BDer(L)$, we define two bilinear maps by 
$$\delta^+(x,y)=\delta(x,y)+\delta(y,x), \hspace{0.3 cm}\delta^-(x,y)=\delta(x,y)-\delta(y,x).$$
It is easy to see that $\delta^+ \in BDer_+(L)$ and $\delta^- \in BDer_-(L)$. Since $\delta =\frac{1}{2}(\delta^+ + \delta^-)$ it follows that
$$BDer(L)=BDer_+(L) \oplus BDer_-(L).$$
To characterize $BDer(L)$, we only need to characterize $BDer_+(L)$ and $BDer_-(L)$.\\
\subsection{Skew-symmetric biderivations}
\begin{theorem}
            \begin{enumerate}
                \item Any skew-symmetric bilinear maps $\delta^-: L_{3,1}\times L_{3,1} \to L_{3,1}$ is a skew-symmetric biderivations of $L_{3,1}$.
                \item Let $\delta: L_{3,2}\times L_{3,2} \to L_{3,2}$ be a skew-symmetric bilinear maps, then $\delta^-$ is a skew-symmetric biderivations of $L_{3,2}$ if and only if there exist $\alpha,\beta,\mu \in \mathbb{F}$ such that 
        $$\delta^-(x,y)=\alpha f_{1,2}(x,y)e_1+\beta f_{1,2}(x,y)e_2+(\mu f_{1,2}(x,y)+\beta f_{1,3}(x,y)-\alpha f_{2,3}(x,y))e_3, \quad \forall x=\sum_{i=1}^3x_ie_i,y=\sum_{i=1}^3y_ie_i \in L_{3,2}.$$
        \item Let $\delta: L_{3,3}\times L_{3,3} \to L_{3,3}$ be a skew-symmetric bilinear maps, then $\delta^-$ is a skew-symmetric biderivations of $L_{3,3}$ if and only if there exist $\alpha,\beta \in \mathbb{F}$ such that 
        $$\delta^-(x,y)=\alpha f_{1,2}(x,y)e_2+(\alpha f_{1,3}(x,y)+\beta f_{1,2}(x,y))e_3, \quad \forall x=\sum_{i=1}^3x_ie_i,y=\sum_{i=1}^3y_ie_i \in L_{3,3}.$$
        \item Let $\delta: L_{3,4}\times L_{3,4} \to L_{3,4}$ be a skew-symmetric bilinear maps, then $\delta^-$ is a skew-symmetric biderivations of $L_{3,4}$ if and only if there exist $\alpha,\beta \in \mathbb{F}$ such that 
        $$\delta^-(x,y)=\alpha f_{1,2}(x,y)e_2+(\alpha f_{1,3}(x,y)+\beta f_{1,2}(x,y))e_3, \quad \forall x=\sum_{i=1}^3x_ie_i,y=\sum_{i=1}^3y_ie_i \in L_{3,4}.$$
            \end{enumerate}
        where $f_{i,j}:L_{3,2}\times L_{3,2} \to \mathbb{F}, (x,y) \mapsto x_iy_j-x_jy_i$.
        \end{theorem}
        \begin{proof}
           For $L_{3,1}$, all binary and ternary products vanish identically. Consequently, any skew-symmetric bilinear map on $L_{3,1}$ constitutes a skew-symmetric biderivation.\\

 For $L_{3,2}$, the 'if' direction is obvious. We now prove the 'only if' direction.\\
 
Let $\delta^-$ be a  skew-symmetric biderivation on  $L_{3,2}$ and $x,y \in L_{3,2}$, such that $x=\sum_{i=1}^{3}x_{i}e_{i}$ and $y=\sum_{i=1}^{3} y_{i}e_{i}$. Then, by the bilinearity of $\delta^-$, we obtain, 
$$\delta^-(x,y)=\sum_{i=1}^{3} \sum_{j=1}^{3} x_{i} y_{j} \delta^-(e_i,e_j)=\sum_{i=1}^{3} \sum_{j=1}^{3} x_{i} y_{j} \delta^-_{e_i}(e_j).$$ 
Since, the linear map $\delta^-_{e_i}:L_{2,3} \to L_{3,2},\, x \mapsto \delta^-_{e_i}(x):=\delta^-(e_i,x)$ is a derivation of $L_{3,2}$, then by Theorem ~\ref{tt0}, the matrix $D_{e_i}$ of $\delta^-_{e_i}$, for $i=1,2,3$ is of the form
\begin{equation}
\label{eq0}
D_{e_i}= \begin{pmatrix}

a_i &b_i&0\\
 c_i& d_i& 0 \\
 p_i & q_i & a_i+d_i
\end{pmatrix}.
\end{equation}
Since $ \delta^-$ is skew-symmetric, then, $\delta(e_i,e_i)=0$ for $i=1,2,3$. Therefore
\begin{equation}
D_{e_1}= \begin{pmatrix}

0 &b_1&0\\
 0& d_1& 0 \\
 0 & q_1 & d_1
\end{pmatrix},\quad D_{e_2}= \begin{pmatrix}

a_2&0&0\\
 c_2& 0& 0 \\
 p_2 & 0 & a_2
\end{pmatrix},\quad D_{e_3}= \begin{pmatrix}

a_3 &b_3&0\\
 c_3& -a_3& 0 \\
 p_3 & q_3 & 0
\end{pmatrix}.
\end{equation}
Also, we have $\delta(e_1,e_i)=-\delta(e_i,e_1)$ for $i=2,3$.  Then, we deduce the following equations:
\begin{gather*}
		  	 \begin{cases}
		  	 \, a_2=-b_1; \hspace{0.3 cm} c_2=-d_1; \hspace{0.3 cm}p_2=-q_1;&\\
		  	 \, a_3=0; \hspace{0.3 cm} c_3=0; \hspace{0.3 cm}p_3=-d_1.	 	  	 
		  	 \end{cases}
\end{gather*}
Furthermore, we have $\delta(e_2,e_3)=-\delta(e_3,e_2)$, we get
\begin{gather*}
\begin{cases}
\, b_3=0; &\\
\, a_3=0; & \\
\, q_3=-a_2.
\end{cases}
\end{gather*}
Therefore, from the above equations, the matrix $D_{e_i}$ for $i=0,1,2,3$, be comes
\begin{equation}
D_{e_1}= \begin{pmatrix}

0 &\alpha&0\\
 0& \beta& 0 \\
 0 & \mu & \beta 
\end{pmatrix},\quad D_{e_2}= \begin{pmatrix}

-\alpha&0&0\\
 -\beta& 0& 0 \\
 -\mu & 0 & -\alpha
\end{pmatrix},\quad D_{e_3}= \begin{pmatrix}

0 &0&0\\
 0& 0& 0 \\
 -\beta & \alpha & 0
\end{pmatrix}.
\end{equation}
where $\alpha=b_1=-a_2; \hspace{0.2 cm}\beta=d_1=-c_2=-p_3 \hspace{0.3 cm} \mu=q_1=-p_2$.
Thus, 
$$\delta^-(x,y)=\sum_{i=1}^3x_iD_{e_i}Y$$
where, 
\begin{center}
$Y=\begin{pmatrix}
		y_1 \\
		y_2\\
		y_3
		\end{pmatrix}$
\end{center}
Then, 
\begin{eqnarray*}
\delta^-(x,y)&=&x_1\begin{pmatrix}
		\alpha y_2 \\
		\beta y_2\\
		\mu y_2+\beta y_3
		\end{pmatrix} -x_2\begin{pmatrix}
		\alpha y_1 \\
		\beta y_1\\
		\mu y_1 +\alpha y_3
		\end{pmatrix}+x_3\begin{pmatrix}
		0 \\
		0\\
		-\beta y_1+\alpha y_2
		\end{pmatrix}\\
        &=&\alpha f_{1,2}(x,y)e_1+\beta f_{1,2}(x,y)e_2+(\mu f_{1,2}(x,y)+\beta f_{1,3}(x,y)-\alpha f_{2,3}(x,y))e_3.
        \end{eqnarray*}
        The cases $L_{3,3}$ and $L_{3,4}$  are treated in the same way. Then, we have the following theorem.
\end{proof}
\subsection{Symmetric biderivations}
In this subsection, we give a description of symmetric biderivations of the algebras $L_{3,i}$, (for $i=1,2,3$). 

While symmetric biderivations are known to vanish for certain classes of algebras (see ~\cite{Baj, ob0, ob2}, they are non-trivial within the context of the Lie-Yamaguti algebras examined in the following subsection, as well as in other recent works ~\cite{ob1}.
\begin{proposition}
   Any symmetric bilinear maps $\delta^+: L_{3,1}\times L_{3,1} \to L_{3,1}$ is a symmetric biderivations of $L_{3,1}$. 
\end{proposition}
\begin{proof}
   Since for $L_{3,1}$, all binary and ternary products vanish identically, then any symmetric bilinear map on $L_{3,1}$ constitutes a symmetric biderivation.  
\end{proof}
\begin{proposition}\label{pp1}
    A bilinear map $\delta^+ : L_{3,2} \times L_{3,2} \rightarrow L_{3,2}$ is a symmetric biderivation if and only if there exists  $\alpha_i , \beta_i, \gamma_i \in \mathbb{F}$ ($i=1,2,3.)$ such that the derivations $\delta^+_{e_i}$ has the following form
     $$D_{e_1}=\begin{pmatrix}
   \alpha_1 &\beta_1&0\\
    \alpha_2&\beta_2&0\\
    \alpha_3&\beta_3&\alpha_1+\beta_2
    \end{pmatrix}\quad D_{e_2}=\begin{pmatrix}
    \beta_1&\gamma_1&0\\
    \beta_2&\gamma_2&0\\
    \beta_3&\gamma_3&\beta_1+\gamma_2
    \end{pmatrix}\quad D_{e_3}=\begin{pmatrix}
    0&0&0\\
    0&0&0\\
    \alpha_1+\beta_2&\beta_1+\gamma_2 &0
    \end{pmatrix}.$$
\end{proposition}
\begin{proof}
    The 'if' direction is obvious. We now prove the 'only if' direction.
 
Let $\delta^+$ be a  skew-symmetric biderivation on  $L_{3,2}$ and $x,y \in L_{3,2}$, such that $x=\sum_{i=1}^{3}x_{i}e_{i}$ and $y=\sum_{i=1}^{3} y_{i}e_{i}$. Then, by the bilinearity of $\delta^+$, we obtain, 
$$\delta^+(x,y)=\sum_{i=1}^{3} \sum_{j=1}^{3} x_{i} y_{j} \delta^+(e_i,e_j)=\sum_{i=1}^{3} \sum_{j=1}^{3} x_{i} y_{j} \delta^+_{e_i}(e_j).$$ 
Since, the linear map $\delta^+_{e_i}:L_{3,2} \to L_{3,2},\, x \mapsto \delta^+_{e_i}(x):=\delta^+(e_i,x)$ is a derivation of $L_{3,2}$, then by Theorem ~\ref{tt0}, the matrix $D_{e_i}$ of $\delta^+_{e_i}$, for $i=1,2,3$ is of the form
\begin{equation}
\label{eq01}
D_{e_i}= \begin{pmatrix}

a_i &b_i&0\\
 c_i& d_i& 0 \\
 p_i & q_i & a_i+d_i
\end{pmatrix}.
\end{equation}
From the equalities $\delta^+(e_1,e_2)=\delta^+(e_2,e_1)$, we have 
\begin{equation}\label{ss1}
    a_2=b_1, \quad c_2=d_1 \quad \mbox{and} \quad p_2=q_1.
\end{equation}
The equalities $\delta^+(e_1,e_3)=\delta^+(e_3,e_1)$, implies that  
\begin{equation}\label{ss2}
    a_3=0, \quad c_3=0 \quad \mbox{and} \quad p_3=a_1+b_1.
\end{equation}
Since $\delta^+(e_2,e_3)=\delta^+(e_3,e_2)$, then 
\begin{equation}\label{ss3}
    b_3=0, \quad d_3=0 \quad \mbox{and} \quad q_3=a_2+d_2.
\end{equation}
Comparing Equations (~\ref{ss1})–(~\ref{ss3}) and setting $\alpha_1=a_1, \, \alpha_2=c_1,\, \alpha_3=p_1,\, \beta_1=b_1,\, \beta_2=d_1, \beta_3=q_1, \, \gamma_1=b_2,\, \gamma_2=d_2, \, \gamma_3=q_2$, we deduce that
 $$D_{e_1}=\begin{pmatrix}
   \alpha_1 &\beta_1&0\\
    \alpha_2&\beta_2&0\\
    \alpha_3&\beta_3&\alpha_1+\beta_2
    \end{pmatrix}\quad D_{e_2}=\begin{pmatrix}
    \beta_1&\gamma_1&0\\
    \beta_2&\gamma_2&0\\
    \beta_3&\gamma_3&\beta_1+\gamma_2
    \end{pmatrix}\quad D_{e_3}=\begin{pmatrix}
    0&0&0\\
    0&0&0\\
    \alpha_1+\beta_2&\beta_1+\gamma_2 &0
    \end{pmatrix}.$$
\end{proof}
\begin{proposition}
    A bilinear map $\delta^+ : L_{3,3} \times L_{3,3} \rightarrow L_{3,3}$ is a symmetric biderivation if and only if there exists  $\alpha_i , \beta_i, \gamma_i \in \mathbb{F}$ ($i=1,2,3.)$ such that the derivations $\delta^+_{e_i}$ has the following form
     $$D_{e_1}=\begin{pmatrix}
   \alpha_1 &0&0\\
    \alpha_2&\beta_1&0\\
    \alpha_3&\beta_2&2\alpha_1+\beta_1
    \end{pmatrix}\quad D_{e_2}=\begin{pmatrix}
    0&&0\\
    \beta_1&\gamma_1&0\\
    \beta_2&\gamma_2&\gamma_1
    \end{pmatrix}\quad D_{e_3}=\begin{pmatrix}
    0&0&0\\
    0&0&0\\
    2\alpha_1+\beta_1&\gamma_1 &0
    \end{pmatrix}.$$
\end{proposition}
\begin{proof}
    The proof proceeds similarly to the proof of Proposition ~\ref{pp1}.

\end{proof}
\begin{proposition}
    A bilinear map $\delta^+ : L_{3,4} \times L_{3,4} \rightarrow L_{3,4}$ is a symmetric biderivation if and only if there exists  $\alpha_i , \beta_i, \gamma_i \in \mathbb{F}$ ($i=1,2,3.)$ such that the derivations $\delta^+_{e_i}$ has the following form
     $$D_{e_1}=\begin{pmatrix}
   0 &0&0\\
    \alpha_1&\beta_1&0\\
    \alpha_2&\beta_2&\beta_1
    \end{pmatrix}\quad D_{e_2}=\begin{pmatrix}
    0&0&0\\
    \beta_1&\gamma_2&0\\
    \beta_2&\gamma_3&\gamma_2
    \end{pmatrix}\quad D_{e_3}=\begin{pmatrix}
    0&0&0\\
    0&0&0\\
    \beta_1&\gamma_2 &0
    \end{pmatrix}.$$
\end{proposition}
\begin{proof}
    The argument is analogous to the proof of Proposition~\ref{pp1}.
\end{proof}
 \section*{Declarations:}

 \noindent \textbf{Disclosure statement:}\\
 The authors declare no conflicts of interest.\\
 \noindent \textbf{Data availability statement:}\\
 No data was used for the research described in the article.\\
 
\noindent \textbf{Ethical Approval :}
not applicable.\\
\noindent \textbf{Funding :}
not applicable.


\begin{thebibliography}{99}


\bibitem{Baj}
Bajo, I., Benayadi, S.,  Oubba, H. (2025). Non-existence of Symmetric Biderivations on Finite-Dimensional Perfect Lie Algebras: I. Bajo et al. Mediterranean Journal of Mathematics, 22(6), 162.
 
\bibitem{bn}
Benayadi, S.  Oubba, H. Nonassociative algebras of biderivation-type. Linear Algebra and its Applications, 2024, vol. 701, p. 22-60.

\bibitem{br0}
 Bresar, M.  Martindale, W. S.  Miers, C. R. Centralizing maps in prime rings with involution, J. Algebra 161, 342-357 (1993).
 \bibitem{br1}
   Bresar, M.  Zhao, K. Biderivations and commuting linear maps on Lie algebras, J. Lie Theory 28,
885-900 (2018).
\bibitem{br}
Bre\v{s}ar, M., Martindale 3rd, W.S., Miers, C.R.:
 Centralizing maps in prime rings with involution. J. Algebra. 161, 342-357 (1993)
\bibitem{du}
 Du, Y. Wang, Y.  Biderivations of generalized matrix algebras, Linear Algebra Appl. 438(11), 4483–
4499 (2013).







\bibitem{Fil1}
Filppov, V.T.  On $\delta$-derivations of Lie algebras. Sib. Math. J. 39(6), 1218–1230 (1998)

\bibitem{Fil2}
Filppov, V.T.  $\delta$-Derivations of prime Lie algebras. Sib. Math. J. 40(1), 174–184 (1999)

\bibitem{Kad}
Kadison R. V., Local derivations. Journal of Algebra, (1990) Vol. 130, P. 494–509

\bibitem{khud}
 Khudoyberdiyev A. K., Yusupov B. B., Local and 2-local $\frac{1}{2}$-derivations on finite-dimensional Lie algebras.Results in Mathematics, (2024) Vol. 79, 210.


\bibitem{Lar}
 Larson D. R., Sourour A. R., Local derivations and local automorphisms of B(X). Proceedings of Symposiain Pure Mathematics, (1990), P. 187–194




\bibitem{ob0}
Oubba, H. Generalized quaternion algebras. Rendiconti del Circolo Matematico di Palermo Series 2, 2023, vol. 72, no 8, p. 4239-4250.
\bibitem{ob1}
Oubba, H. Biderivations, commuting linear maps, post-Lie algebra structure on solvable Lie algebras of maximal rank. Communications in Algebra, 2024, p. 1-10.
\bibitem{ob2}
Oubba, H. Local (2-Local) derivations and automorphisms and biderivations of complex $\omega$-Lie algebras. Le Matematiche, 2024, vol. 79, no 1, p. 135-150.



\bibitem{Mos}
Mosbahi, Bouzid, and Yusupov Bakhtiyor. "Derivations, local and 2-local derivations and automorphism groups of 3-dimensional nilpotent Lie-Yamaguti algebras." Communications in Algebra (2026): 1-17.










 
\bibitem{Ab}
Abdelwahab, H., Barreiro, E., Calderón, A. J., Ouaridi, A. F. (2022). The classification of nilpotent Lie-Yamaguti
algebras. Linear Algebra Appl. 654:339–378. DOI: 10.1016/j.laa.2022.09.005.







\bibitem{Chen}
Chen, Z.:
 Biderivations and linear commuting maps on simple generalized Witt algebras over a field. Electron. J. Linear Algebra. 31, 1-12 (2016)
 
\bibitem{mam}
Mamadaliyev  U.  K.,  Sattarov  A.,  Yusupov  B.  B.,  Local  and  2-local $\frac{1}{2}$-derivations  of  solvable  Leibnizalgebras. Eurasian Mathematical Journal, (2025) Vol. 16, Iss. 2, P. 42–54.



\bibitem{Gua}
Gua. X., Liu. X., Zhao, K.:
 Biderivations of Block Lie algebras, Linear Algebra Appl. 538, 43-55 (2018)

\bibitem{Han}
Han, X., Wang, D., Xia, C.:
  Linear commuting maps and biderivations on the Lie algebra $W(a,\,b)$. J. Lie Theory 26, 777-786 (2016)

\bibitem{Sem}
Semrl  P.,  Local  automorphisms  and  derivations  on B(H).  Proceedings  of  the  American  MathematicalSociety, (1997) Vol. 125, P. 2677–2680.

  
\bibitem{Yus}
 Yusupov B., Vaisova N., Madrakhimov T., Local $\frac{1}{2}$-derivation on n-dimensional naturally graded quasi-filiform Leibniz algebra of type I. AIP Conference Proceedings, (2024)
\bibitem{Yus1}
 Yusupov  B.,  Local $\frac{1}{2}$-derivation  on  naturally  graded  filiform  Leibniz  algebra.  Journal  of  MathematicalSciences, (2025),
 \bibitem{Yus2}
 Yusupov B., Local and 2-local $\frac{1}{2}$-derivation on naturally graded non-Lie p-filiform Leibniz algebras. Journalof Algebra Combinatorics Discrete Structures and Applications, (2025) Vol. 12, Iss. 3, P. 275–288.

\bibitem{Tang}
Tang, X.:
 Biderivations of finite dimensional complex simple Lie algebras. Linear and Multilinear Algebra. 66, 250-259 (2018)
\bibitem{Wang1}
Wang, D., Yu, X.:
 Biderivations and Linear commuting maps on the Schr\"{o}dinger-Virasoro Lie algebras , Comm. Algebra. 41, 2166-2173 (2013)
\bibitem{Wang2}
Wang, D., Yu, X., Chen, Z.:
 Biderivations of the parabolic subalgebra of simple Lie algebras, Comm. Algebra. 39, 4097-4104 (2011)

\end{thebibliography}
\end{document}